\documentclass[a4paper,12pt]{article}
\usepackage{setspace}

\usepackage{latexsym} 

\usepackage[utf8]{inputenc}
\usepackage[english]{babel}

\usepackage{amsfonts,amsmath,amssymb,amsbsy,amsthm,enumerate}

\usepackage[mathscr]{euscript}
\usepackage{pstricks}
\usepackage{multirow}
\usepackage{verbatim}
\usepackage{color}
\usepackage{lineno}
\usepackage{hyperref}
\usepackage{subcaption}

\newtheorem{theorem}{Theorem}

\newtheorem{lemma}{Lemma}
\newtheorem{corollary}[theorem]{Corollary}

\usepackage{tikz}
\usepackage[subrefformat=parens,labelformat=parens]{subcaption}
\usetikzlibrary{calc,positioning,decorations.pathmorphing,decorations.pathreplacing}
\tikzset{black vertex/.style={circle,draw,minimum size=1mm,inner sep=0pt,outer sep=2pt,fill=black, color=black}}
\tikzset{largepath/.style={orange!75!white,line width=6pt,line cap=round,opacity=0.5}}

\usepackage{graphicx}

\title{Ramsey Goodness of paths and unbalanced graphs\footnote{
This research has been partially supported by Coordena\c cão de Aperfei\c coamento
de Pessoal de N\'\i vel Superior -- Brasil -- CAPES -- Finance Code 001.
F.~Botler is supported by CNPq {(\small 304315/2022-2)} and
CAPES {(\small 88887.878880/2023-00)}.
L. Moreira is supported by FAPESB {(\small APP0044/2023)}.
CNPq is the National Council for Scientific and Technological Development of Brazil.
FAPESB is the Bahia Research Foundation.\\
E-mail addresses: fbotler@ime.usp.br (F. Botler), lpfm@dmat.ufpe.br (L. Moreira), jpsouza@cos.ufrj.br (J. Souza)
}}
\author{Fábio Botler$^1$ \and Luiz Moreira$^2$ \and João Pedro de Souza$^{3,4}$}
\date{
$^1$Universidade de São Paulo\\%
$^2$Universidade Federal de Pernambuco\\%
$^3$Universidade Federal do Rio de Janeiro\\%
$^4$Colégio Pedro II
}

\begin{document}
\maketitle

\begin{abstract}
    Given graphs \(G\) and \(H\),
    we say that \(G\) is \emph{\(H\)-good} if the Ramsey number \(R(G,H)\)
    equals the trivial lower bound \((|G| - 1)(\chi(H) - 1) + \sigma(H)\),
    where \(\chi(H)\) denotes the usual chromatic number of \(H\),
    and \(\sigma(H)\) denotes the minimum size of a color class in a \(\chi(H)\)-coloring of \(H\).
    Pokrovskiy and Sudakov [Ramsey goodness of paths. Journal of Combinatorial Theory, Series B, 122:384–390, 2017.] proved that \(P_n\) is \(H\)-good whenever \(n\geq 4|H|\).
    In this paper, given \(\varepsilon>0\), we show that if \(H\) satisfy a special unbalance condition,
    then \(P_n\) is \(H\)-good whenever \(n \geq (2 + \varepsilon)|H|\).
    More specifically, we show that if \(m_1,\ldots, m_k\) are such that \(\varepsilon\cdot m_i \geq 2m_{i-1}^2\) for \(2\leq i\leq k\),
    and \(n \geq (2 + \varepsilon)(m_1 + \cdots + m_k)\), 
    then \(P_n\) is \(K_{m_1,\ldots,m_k}\)-good.
\end{abstract}

In this paper, we consider only finite and undirected graphs without loops or multiple edges.
Throughout this text, given a graph \(G\), 
we denote by \(|G|\) the number of vertices of~\(G\).
Given graphs \(K\), \(G\) and \(H\), we write \(K \rightarrow (G,H)\) if every red--blue coloring of the edges of \(K\) contains a red copy of \(G\) or a blue copy of \(H\); and the \emph{Ramsey number} \(R(G,H)\) is the minimum positive integer \(N\) for which \(K_N \rightarrow (G,H)\).
Understanding the behavior of \(R(G,H)\) is one of the main problems in Extremal Combinatorics, 
and have been extensively explored over the last century~\cite{conlon2015recent} while significant results were obtained recently~\cite{campos2023exponential,conlon2018ramsey,fiz2020triangle,mattheus2024asymptotics,sah2023diagonal}.

A natural lower bound for \(R(G,H)\) is as follows.
Let \(\chi(H)\) be the chromatic number of~\(H\), i.e., the smallest number of colors for which there is a proper coloring of the vertices of \(H\),
and let \(\sigma(H)\) be the minimum size of a smallest class in a minimum proper coloring of~\(H\),
i.e., \(\sigma(H) =\min \{|c^{-1}(i)| : c \text{ is a proper coloring of } H \text{ with } \chi(H) \text{ colors}, i\in\big[\chi(H)\big]\}\).
Burr~\cite{burr1981ramsey} observed that if \(G\) is a connected graph and \(|G|\geq \sigma(H)\), then we have
\begin{equation}\label{lower-bound:trivial}
    R(G,H) \geq (|G|-1)(\chi(H)-1)+\sigma(H).
\end{equation}
Indeed, put \(N=(|G|-1)(\chi(H)-1)+\sigma(H)-1\),
and consider the red--blue coloring of \(E(K_N)\) obtained from \(\chi(H)-1\) disjoint red cliques with \(|G|-1\) vertices and one red clique with \(\sigma(H)-1\) vertices by coloring the remaining edges blue.
Such a coloring contains no red copy of \(G\) because each red component has size at most \(|G|-1\);
and contains no blue copy of \(H\) because 
the blue edges induce a \(\chi(H)\)-partite graph \(K^B\),
but  different parts of \(H\) must fit in different parts of \(K^B\),
while no part of \(H\) fits in the part of size \(\sigma(H)-1\).

Motivated by this construction, in a seminal paper, Erd\H{o}s and Burr~\cite{burr1983generalizations} introduced the concept of \emph{goodness}.
More specifically, we say that a graph $G$ is \emph{\(H\)-good} if~\eqref{lower-bound:trivial} holds with equality, i.e., whenever \(R(G,H)=(|G|-1)\big(\chi(H)-1\big)+\sigma(H)\).
Erd\H{o}s and Burr~\cite{burr1983generalizations} 
presented several examples of $K_r$-good sparse graphs, 
and posed a number of questions that shaped this subarea since then.
The main such question is whether bounded degree are graphs $K_r$-good.

Although Brandt disproved the main conjecture~\cite{brandt2006global}, goodness of sparse graphs have been extensively studied.
Nikiforov and Rousseau~\cite{NiRo2009} answered all of other questions,
and, in particular, their results together with the Separator Theorem of Alon, Seymour, and Thomas~\cite{AST1990} imply that all planar graphs are $K_r$-good.
More recently, Allen, Brightwell, and Skokan~\cite{ABS}, among other interesting results, 
proved that bounded degree graphs with sublinear bandwidth are \(K_r\)-good. 

In a more general setup,
\(H\)-goodness is known for some specific classes of sparse graphs and any fixed graph \(H\), not necessarily complete.
Erd{\H{o}}s, Faudree, Rousseau and Schelp~\cite{EFRS} proved that every bounded degree tree with $n$ vertices is $H$-good for every fixed $H$ and sufficiently large $n \in \mathbb{N}$, and their result was strengthened first by Balla, Pokrovskiy and
Sudakov~\cite{BPS2018}, and recently by Montgomery, Pavez-Signé, and Yan~\cite{montgomery2023ramsey} who proved that if $n = \Omega(|H|)$ then every bounded degree tree with \(n\) vertices is $H$-good.
In the special case of paths, in 2017, Pokrovskiy and Sudakov~\cite{pokrovskiy2017ramsey} presented a much tighter result that \(P_n\) is \(H\)-good whenever \(n\geq 4|H|\), where \(P_n\) denotes the path with \(n\) vertices.
Observe that it suffices to verify the case \(H\) is a complete multipartite graph.

\begin{theorem}[Pokrovskiy--Sudakov, 2017]\label{thm:pok.1}
Given integers \(m_1\leq m_2 \leq \cdots \leq m_k\), \(n\geq 3m_k+5m_{k-1}\) and \(N\geq(n-1)(k-1)+m_1\), we have \(K_N\rightarrow(P_n,K_{m_1,\ldots,m_k})\).
\end{theorem}

Pokrovskiy and Sudakov~\cite{pokrovskiy2017ramsey}
also observe that the multiplicative constant (\(4\)) cannot be reduced below \(2\),
giving the following example which we present here for completeness.
Fix \(H = K_{m_1,m_2}\) with \(m_1 \leq m_2\),
and suppose \(n = 2m_2 - 2< 2(m_1 + m_2) = 2|H|\).
Then let \(N = n-1 + m_1\) and color \(K_N\)
with disjoint (curiously) blue cliques of size \(m_1 + m_2 - 1\) and \(m_2 - 2\),
joined by red edges.
Such a coloring has no blue copy of \(H\) because each blue component has size at most \(|H| - 1\);
and has no red copy of \(P_n\) because any red path must alternate between the two blue cliques, and hence has size at most \(2m_2 - 3\).

More recently, Pokrovskiy and Sudakov~\cite{pokrovskiy2020ramsey}
obtained a result on Ramsey goodness of cycles 
by imposing additional conditions on the sizes the color classes of \(H\).

\begin{theorem}[Pokrovskiy--Sudakov, 2020]\label{thm:pok.2}
If \(n\geq10^{60}m_k\), \(m_1\leq m_2\leq \cdots\leq m_k\) satisfying \(m_i\geq i^{22}\), for each \(i\), and \(N\geq (n-1)(k-1)+m_1\), then \(K_N\rightarrow(C_n,K_{m_1,\ldots,m_k})\).
\end{theorem}

When we consider a simple unbalance condition in the size of the largest parts,
we obtain the following straightforward consequence of Theorem~\ref{thm:pok.1}
which brings the constant \(4\) closer to \(3\).

\begin{corollary}\label{cor:pok.1}
Let \(\varepsilon \in (0,1]\).
Given integers \(m_1\leq m_2 \leq \cdots \leq m_k\), such that \(\varepsilon \cdot m_k \geq 2m_{k-1}\) and let \(n\geq (3+ \varepsilon)(m_1+ \cdots + m_k)\) and \(N\geq(n-1)(k-1)+m_1\). 
Then \(K_N\rightarrow(P_n,K_{m_1,\ldots,m_k})\).
\end{corollary}

\begin{proof}
    Since \(\varepsilon\cdot m_k \geq 2m_{k-1}\), we have
\[
    n \geq (3+\varepsilon)(m_k+m_{k-1}) 
                \geq 3m_k+\varepsilon\cdot m_k +3m_{k-1} 
                 \geq 3m_k+5m_{k-1}.
\]
Therefore, by Theorem~\ref{thm:pok.1}, \(K_N\rightarrow(P_n,K_{m_1,\ldots,m_k})\) as desired.
\end{proof}

In this paper, we reduce the constant \(3+\varepsilon\) to \(2+\varepsilon\)
by imposing a stronger unbalance condition on \(H\).
Given the example presented above, this result is somehow tight.

\begin{theorem}\label{thm:main}
Let \(\varepsilon \in (0,1]\)
and let \(m_1\leq m_2 \leq \cdots \leq m_k\) be  positive integers such that \(\varepsilon\cdot m_i\geq 2m_{i-1}^2\) for every \(2\leq i\leq k\).
If \(n\geq (2+ \varepsilon)(m_1+ \cdots + m_k)\) and \(N\geq (n-1)(k-1)+m_1\),
then \(K_N\rightarrow(P_n,K_{m_1,\ldots,m_k})\).
\end{theorem}

Observe that both Corollary~\ref{cor:pok.1} and Theorem~\ref{thm:main}
imply that \(K_N \rightarrow(P_n,H)\) whenever \(H\) admits a coloring
with color classes satisfying their respective unbalance conditions.

Our technique borrows a few ideas from~\cite{pokrovskiy2017ramsey},
but we make only a mild use of P\'osa rotation-extention technique (see Lemma~\ref{lemma:brandt}).
We, alternatively, explore the structure of the graph obtained when removing a longest path,
as well as relations between the neighbors of the vertices of such graph in the path (see Section~\ref{sec:auxiliary}).

\bigskip
\noindent\textbf{Organization of the paper.}
In Section~\ref{sec:auxiliary}, we present some auxiliary results that could be useful in a more general context;
in Section~\ref{sec:bipartite} we verify Theorem~\ref{thm:main} in the special case \(k=2\) (see Theorem~\ref{thm:main-bipartite});
and in Section~\ref{sec:general} we verify Theorem~\ref{thm:main}.
For ease of notation, when dealing with bipartite graphs, we use \(s,t\) instead of \(m_1,m_2\) for the size of its parts.

\section{Auxiliary results}\label{sec:auxiliary}

Given a graph \(G\) and disjoint sets \(X, Y \subseteq V(G)\),
we denote by \(G[X,Y]\) the bipartite graph with vertex set \(X\cup Y\)
and whose edges are the edges of \(G\) that join a vertex of \(X\)
to a vertex of \(Y\);
and by \(e(X,Y)\) the number of edges of \(G[X,Y]\).

In our proof we use the following result to bound
the order of a longest path in the red graph.
Such a result can be proved with a simple depth-first search (see, e.g.,~\cite[Lemma 3.3]{moreira2021ramsey}).

\begin{lemma}\label{lemma:dfs-1}
	Let \(s,t\) be two positive integers,
	and let \(G\) be a graph on \(N\) vertices.
    Then at least one of the following holds: 
    (i) there is a pair \(S\), \(T\) of disjoint sets of vertices
    with \(|S| = s\) and \(|T| = t\) for which \(e(S,T) = 0\); or
    (ii) \(G\) contains a path of order \(N - s - t + 1\).
\end{lemma}

The following result on split graphs is also useful to find a blue copy of \(K_{s,t}\).

\begin{lemma}\label{lemma:split}
	Let \(s,t\) be two positive integers with \(s\leq t\),
	and let \(G\) be a graph obtained from a clique \(X\) and a graph with vertex set \(Y\) with \(X\cap Y = \emptyset\) 
	by joining every vertex of \(X\) to every vertex of \(Y\).
	If \(|X| \geq s\) and \(|X|+|Y| \geq s+t\),
	then \(G\) contains a \(K_{s,t}\).
\end{lemma}

\begin{proof}
	Let \(S\subseteq X\) be a set of size \(s\)
	and \(T \subseteq (X\cup Y)\setminus S\) be a set of size \(t\),
    then \(G[S,T]\) is the desired \(K_{s,t}\).
\end{proof}

Given a graph \(G\) and a set \(X\subseteq V(G)\),
we denote by \(N_G(X)\) the set of vertices in \(V(G)\setminus X\)
that are adjacent to at least one vertex of \(X\).
We omit subscripts when it is clear from the context.
A \emph{component} of a graph \(G\) is a set \(F\) of vertices of \(G\) 
for which \(G[F]\) is a maximal connected subgraph of \(G\).
Now, consider a path \(P\subseteq G\).
In what follows, we fix an ordering \(P = u_0 \cdots u_\ell\) of \(V(P)\).
Let \(F\) be a component of \(G\setminus V(P)\),
and observe that \(N(F) \subseteq V(P)\).
We denote by \(N(F)^-\) the set \(\{u_{i-1} : u_i \in N(F)\}\).
Also, given vertices \(u\) and \(u'\) in \(P\),
we denote by \(uPu'\) the subpath of \(P\) joining \(u\) and \(u'\).
When \(P\) is a longest path in \(G\),
we obtain the following.

\begin{lemma}\label{lemma:blue-clique}
    Let \(G\) be a graph, let \(P\) be a longest path in \(G\),
    and let \(F\) be a component of \(G\setminus V(P)\).
    Then (i) \(N(F)^- \cap N(F) = \emptyset\); and (ii) \(N(F)^-\) is an independent set.
\end{lemma}

\begin{proof}
Let \(P = u_0 \cdots u_\ell\) be as above.
(i)~If \(u_i \in N(F)^-\cap N(F)\), then \(u_i,u_{i+1}\in N(F)\).
Let \(Q\) be a path joining \(u_i\) to \(u_{i+1}\) whose internal vertices are in \(F\).
Then \((P\cup Q) - u_iu_{i+1}\) is a path in \(G\) with at least \(|P|+1\) vertices, a contradiction to the maximality of \(P\).
(ii)~By the maximality of \(P\) we have \(u_0\notin N(F)\).
Let \(u_i, u_j \in N(F)\) with \(1 \leq i < j\) and suppose that \(u_{i-1}u_{j-1} \in E(G)\).
Let \(Q\) be a path joining \(u_i\) to \(u_j\)
whose internal vertices are in \(F\).
Then \((P\cup Q) - u_{i-1}u_i - u_{j-1}u_j + u_{i-1}u_{j-1}\)
is a path in \(G\) with at least \(|P|+1\) vertices (see Figure~\ref{fig:blue-clique}),
a contradiction to the maximality of~\(P\). 
\end{proof}

The next lemma gives a bound on the number of vertices of \(N(F)^-\) that can be 
adjacent to a vertex in \(F'\) for any two distinct components \(F\) and \(F'\) of \(G\setminus V(P)\).
Its proof is analogous to the proof above.
We include it for completeness.

\begin{lemma}\label{lemma:abab}
    Let \(G\) be a graph, let \(P\) be a longest path in \(G\),
    and let \(F\) and \(F'\) be distinct components of \(G\setminus V(P)\).
    Then \(|N(F)^- \cap N(F')| \leq 1\).
\end{lemma}

\begin{proof}
    Let \(P = u_0 \cdots u_\ell\) be as above.
    Suppose there are two distinct vertices \(u_i,u_j\in N(F)^- \cap N(F')\).
    By the definition of \(N(F)^-\), we have \(u_{i+1},u_{j+1} \in N(F)\).
    Let \(Q\) (resp. \(Q'\)) be a path in \(G\) connecting \(u_{i+1}\) to \(u_{j+1}\) (resp. \(u_i\) to \(u_j\)) whose internal vertices are in \(F\) (resp. in \(F'\)).
    Then \(P^* = \big(P\cup Q \cup Q'\big) - u_iu_{i+1} - u_ju_{j+1}\)
    is a path in \(G\) with at least \(|P| + 2\) vertices (see Figure~\ref{fig:abab}),
    a contradiction to the maximality of \(P\).
\end{proof}

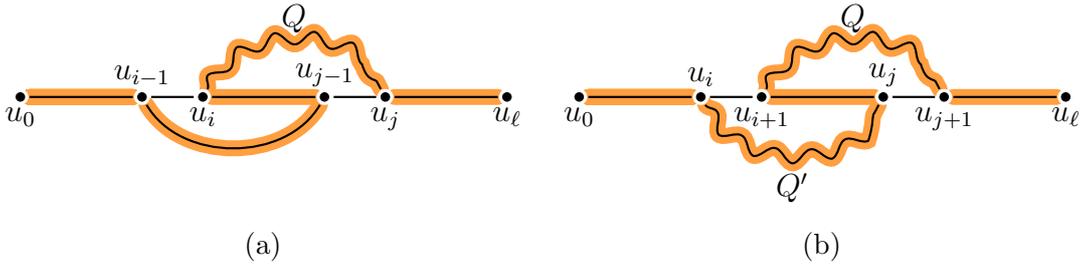
\begin{figure}[h]
    \centering
    \begin{subfigure}{0.45\textwidth}
        \centering
        \begin{tikzpicture}[scale = 0.8]

            \foreach \nn in {0,2,3,5,6,8}{
                \node (\nn) [black vertex] at (\nn,0) {};
            }

            \node [anchor=north] at (0) {$u_0$};
            \node [anchor=south] at (2) {$u_{i-1}$};
            \node [anchor=north] at (3) {$u_{i}$};
            \node [anchor=south] at (5) {$u_{j-1}$};
            \node [anchor=north] at (6) {$u_{j}$};
            \node [anchor=north] at (8) {$u_\ell$};
            \node [] at (4.5,1.3) {$Q$};
            \node [] at (3.5,-1.5) {\color{white}$Q'$};

      \draw [largepath] 
      (2) -- (0) -- (2) -- (0)
      (2) to  [bend right = 60]  (5) 
      (5) to (3) 
      (6) to (8);
      \draw[largepath,decorate, decoration=snake, segment length=5mm] (3) to [bend left = 60] (6) ;

            \draw [thick] (0) -- (2) -- (3) -- (5) -- (6) -- (8);
            \draw [thick] (2) to [bend right = 60] (5);
            \draw [thick,decorate, decoration=snake, segment length=5mm] (3) to [bend left = 60] (6);

        \foreach \nn in {0,2,3,5,6,8}{
            \fill [very thick,draw=white,fill=black] (\nn,0) circle (3pt);
        }
        \end{tikzpicture}
    \caption{}
    \label{fig:blue-clique}
    \end{subfigure}
    \begin{subfigure}{0.45\textwidth}
        \centering
        \begin{tikzpicture}[scale = 0.8]

            \foreach \nn in {0,2,3,5,6,8}{
                \node (\nn) [black vertex] at (\nn,0) {};
            }

            \node [anchor=north] at (0) {$u_0$};
            \node [anchor=south] at (2) {$u_{i}$};
            \node [anchor=north] at (3) {$u_{i+1}$};
            \node [anchor=south] at (5) {$u_{j}$};
            \node [anchor=north] at (6) {$u_{j+1}$};
            \node [anchor=north] at (8) {$u_\ell$};
            \node [] at (4.5,1.3) {$Q$};
            \node [] at (3.5,-1.5) {$Q'$};

      \draw [largepath] 
      (2) -- (0) -- (2) -- (0)
      (5) to (3) 
      (6) to (8);

        \draw[largepath,decorate, decoration=snake, segment length=5mm] (2) to  [bend right = 60]  (5) ;
        \draw[largepath,decorate, decoration=snake, segment length=5mm] (3) to [bend left = 60] (6) ;

        \draw [thick] (0) -- (2) -- (3) -- (5) -- (6) -- (8);
        \draw [thick,decorate, decoration=snake, segment length=5mm] (2) to [bend right = 60] (5);
        \draw [thick,decorate, decoration=snake, segment length=5mm] (3) to [bend left = 60] (6);

        \foreach \nn in {0,2,3,5,6,8}{
            \fill [very thick,draw=white,fill=black] (\nn,0) circle (3pt);
        }
        \end{tikzpicture}
    \caption{}
    \label{fig:abab}
    \end{subfigure}
  \caption{
    Left: A longer path obtained in the case $N(F)^-$ is not an independent set;
    Right: A longer path obtained in the case \(|N(F)^- \cap N(F')| \geq 2\).}
\end{figure}

\section{Paths versus unbalanced bipartite graphs}\label{sec:bipartite}

For organizational purposes we first verify Theorem~\ref{thm:main} in the case of bipartite graphs (see Theorem~\ref{thm:main-bipartite}), i.e., when \(k= 2\),
to serve as the base case of our main induction argument.
For that we prove the following lemma.

\begin{lemma}\label{lemma:asymmetric}
    Let \(s\), \(t\), and \(r\) be positive integers with \(r \leq s \leq t\), 
    and let \(G\) be a graph with order \(N\geq 2t+s^2-1\).
    Let \(P\subseteq G\) be a longest path in \(G\).
    If there are \(r\) components in \(G\setminus V(P)\) whose union contains at least \(s\) vertices, then  \(K_{s,t}\subseteq \overline{G}\).
\end{lemma}

\begin{proof}
Let \(G' = G\setminus V(P)\) and let \(F_1,\ldots,F_r\) be components of \(G'\)
with \(|F_1\cup \cdots \cup F_r| \geq s\).
The proof follows by induction on \(r\).
First, suppose \(r = 1\),
i.e., that \(G'\) has one component \(F_1\) of order at least \(s\).
Note that no vertex in \(V(G)\setminus \big(F_1 \cup N_P(F_1)\big)\) is adjacent to a vertex in \(F_1\).
Thus, if \(N - |F_1| - |N(F_1)| \geq t\),
then \(K_{s,t}\subseteq \overline{G}[F_1, V(G)\setminus \big(F_1 \cup N_P(F_1)\big)]\) as desired.
Therefore, we may assume
\begin{equation}\label{eq:lower:F+NF}
    |F_1| + |N(F_1)| \geq N - (t-1) \geq t + s^2 \geq t + s.
\end{equation}
Now, suppose that \(|N(F_1)| \geq s\).
Note that, by the maximality of \(P\), the end vertices of \(P\) are not adjacent to vertices of \(F_1\).
Thus, we have \(|N_P(F_1)^-| = |N_P(F_1)| = |N(F_1)| \geq s\).
By Lemma~\ref{lemma:blue-clique}, no vertex of \(N_P(F_1)^-\) is adjacent to a vertex of \(F_1\), and \(N_P(F_1)^-\) is an independent set.
Thus,~\eqref{eq:lower:F+NF} implies that \(\overline{G}[N_P(F_1)^-\cup F_1]\) is a graph as in the statement of Lemma~\ref{lemma:split},
and hence \(K_{s,t}\subseteq \overline{G}[N_P(F_1)^- \cup F_1]\), as desired.

Therefore we may assume that \(|N(F_1)| \leq s - 1\).
Together with~\eqref{eq:lower:F+NF}, this implies that
\begin{equation}\label{eq:lower:F}
    |F_1| \geq t + s(s-1) + 1 \geq t + s.
\end{equation}
Observe that \(|P| + |F_1| \leq N\),
and, by~\eqref{eq:lower:F}, we have \(|P| \leq N - s - t\).
Therefore, by Lemma~\ref{lemma:dfs-1}, we have \(K_{s,t}\subseteq\overline{G}\), as desired.
This proves the case \(r = 1\).

Now, suppose  \(r \geq 2\) and let \(F^*=F_1\cup \cdots \cup F_r\). 
Although we can get a slightly better upper bound on \(|F^*|\),
here we prove that \(|F^*| \leq s^2 - r(r-1)\).
Indeed, if \(|F_i| \geq s - (r - 2)\) for some \(i\in [r]\),
then we can pick \(F_i\) and \(r-2\) other components of \(G\setminus V(P)\)
and obtain \(r-1\) components whose union contains \(s\) vertices.
But then, by the induction hypothesis, we have \(K_{s,t}\subseteq \overline{G}\).
Thus, we may assume \(|F_i| \leq s - (r-1)\) for every \(i\in [r]\).
Therefore, since \(r \leq s\), we have
    \begin{equation}\label{eq:lemma1_O*_upper_bound_general}
        |F^*|\leq r \big(s - (r-1)\big) = r\cdot s - r(r-1) \leq s^2 - r(r-1),
    \end{equation}
as desired.

Now, analogously to the case \(r = 1\), if \(N-|F^*|-|N(F^*)|\geq t \), then \(K_{s,t}\subseteq \overline{G}[F^*,V(G)\setminus\big(F^*\cup N(F^*)\big)]\), as desired.
Thus, we may assume that  
\begin{equation}\label{eq:lemma1_N(O)_lower_bound_general}
    N-|F^*|-|N(F^*)|\leq t-1.
\end{equation}

Summing~\eqref{eq:lemma1_O*_upper_bound_general} and~\eqref{eq:lemma1_N(O)_lower_bound_general},
by the hypothesis in \(N\),
we obtain 
\[
    |N(F^*)| \geq N - t + 1 + r(r-1) -s^2 \geq t + r(r-1).
\]
Now, by Lemma~\ref{lemma:abab}
given \(i\) and \(j\) with \(i\neq j\),
\(N_P(F_i)^-\) has at most one vertex in \(N(F_j)\).
Consequently, \(N_P(F_i)^-\) has at most \(r-1\) vertices of \(N(F^*)\).
Thus, \(N_P(F^*)^- = \cup_{i=1}^r N_P(F_i)^-\) has at most \(r(r-1)\) vertices of \(N(F^*)\).
Therefore,  there is a set \(N^*\subseteq N_P^-(F^*)\setminus N(F^*)\) with  \(t\) vertices, 
and hence \(K_{s,t}\subseteq \overline{G}[F^*,N^*]\), as desired.
\end{proof}

Now, we can prove Theorem~\ref{thm:main-bipartite},
which requires a slightly weaker unbalance condition than Theorem~\ref{thm:main}.

\begin{theorem}\label{thm:main-bipartite}
    Let \(\varepsilon \in (0,1]\),
    and let \(s \leq t\) be positive integers such that \(\varepsilon \cdot t \geq s^2 - (3+\varepsilon)s\). If \(n \geq (2+\varepsilon)(s+t)\) and \(N \geq (n-1) + s\), then \(K_N \rightarrow (P_n,K_{s,t})\).
\end{theorem}

\begin{proof}
Fix a red--blue coloring of \(K_N\)
and let \(G\) be the graph induced by its red edges.
Let \(P\subseteq G\) be a longest path in \(G\).
If \(N-|P|\leq s-1\), then \(|P|\geq n\), as desired.
Thus,  we may assume that \(N-|P|\geq s\). 
However, as \(N = (2+\varepsilon)(s+t) -1 + s = 2t + (3+\varepsilon)s + \varepsilon \cdot t - 1 \geq 2t + s^2 - 1\),
applying the Lemma~\ref{lemma:asymmetric} with  \(r=s\)
we have \(K_{s,t} \subseteq \overline{G}\), as desired.
\end{proof}

\section{Paths versus unbalanced graphs}\label{sec:general}

To prove the main result of this section we need the following consequence of the well-known P\'osa rotation-extention technique~\cite{brandt2006global,posa1976hamiltonian},
which we restate for our purposes (for a proof see, e.g.,~\cite{bonamy2023separating}).

\begin{lemma}[\cite{brandt2006global}]\label{lemma:brandt}
  Let  \(P\) be a longest path of a graph $G$.
  Then there is a set \(S\subseteq V(P)\)
  for which \(N(S) \subseteq V(P)\) and \(|N(S)| \leq 2|S|\).
\end{lemma}

Now we can prove Theorem~\ref{thm:main}.

\begin{proof}[Proof of Theorem~\ref{thm:main}]
    The proof follows by induction on \(k\). 
    If \(k = 2\) the results follows from Theorem~\ref{thm:main-bipartite}.
    Thus we may assume \(k\geq 3\),
    and that the statement holds for \(k' < k\).
    In this proof, we use the induction hypothesis either
    with \(m_1,\ldots,m_{k-1}\) or with \(m_2,\ldots,m_k\).
    In either case we have \(\varepsilon \cdot m_i \geq 2m_{i-1}^2\) 
    for every \(i\) in the considered interval.
    
    Fix a red--blue coloring of \(K_N\),
    and let \(G\) be the graph induced by its red edges.
    Let \(P\) be a longest path in \(G\).
    If \(|P| \geq n\), then the statement follows.
    Thus, we may assume that \(|P|\leq n-1\). 
    We first deal with the case \(P\) is small.
    
    Suppose that \(|P| \leq n - m_2 - 2m_1 - 1\).
    By Lemma~\ref{lemma:brandt} there is a set \(S\subseteq V(P)\) such that
    \(N(S) \subseteq V(P)\) and \(|N(S)| \leq 2|S|\).
    Thus, \(N - |S\cup N(S)| \geq N - |P| \geq (n-1)(k-2) + m_2\).
    By induction hypothesis either \(P_n\subseteq G\setminus \big(S\cup N(S)\big)\) 
    or \(K_{m_2,\ldots,m_k} \subseteq \overline{G \setminus \big(S\cup N(S)\big)}\).
    In the former case we have \(P_n \subseteq G\) as desired.
    Thus, we may assume that there is a copy \(K'\) of \(K_{m_2,\ldots,m_k}\) in \(\overline{G \setminus \big(S\cup N(S)\big)}\).
    If \(|S| \geq m_1\),
    then \(S\) together with \(V(K')\) induces a copy of \(K_{m_1,\ldots,m_k}\) in \(\overline{G}\), as desired.
    Therefore, we may assume that \(|S| \leq m_1 - 1\).
    Now, let \(A\) be a maximum set in \(V(G)\) such that \(|A| \leq m_1 - 1\)
    and \(|N(A)| \leq 2|A|\).
    Let \(G' = G \setminus \big(A\cup N(A)\big)\) and set \(N' = |V(G')| \geq N - 3m_1\).
    Note that for every set \(A'\subseteq V(G')\) with \(|N_{G'}(A')| \leq 2 |A'|\) we have  \(|N_G(A\cup A')| = |N_G(A)| + |N_{G'}(A')| \leq 2\big(|A| + |A'|\big)\),
    and hence, by the maximality of \(A\), we have \(|A| + |A'| \geq m_1\).
    
    Now, let \(P'\subseteq G'\) be a longest path.
    Again, by Lemma~\ref{lemma:brandt} there is a set \(S'\subseteq V(P')\) such that
    \(N_{G'}(S') \subseteq V(P')\) and \(|N_{G'}(S')| \leq 2|S'|\).
    Since \(|N_{G'}(S')| \leq 2|S'|\), we have \(|A| + |S'|  \geq m_1\).
    Naturally, we have \(|P'| \leq |P| \leq  n - m_2 - 2m_1 - 1\),
    and since \(S' \cup N_{G'}(S') \subseteq V(P')\),
    we have 
    \(N' - |S'\cup N_{G'}(S')| \geq N - 3m_1 - |P| \geq (n-1)(k-2) + m_2\).
    By induction hypothesis either \(P_n\subseteq G'\setminus \big(S'\cup N_{G'}(S')\big)\) 
    or \(K_{m_2,\ldots,m_k} \subseteq \overline{G' \setminus \big(S'\cup N_{G'}(S')\big)}\).
    In the former case we have \(P_n \subseteq G\) as desired.
    Thus, we may assume that there is a copy \(K'\) of \(K_{m_2,\ldots,m_k}\) in \(\overline{G' \setminus \big(S'\cup N_{G'}(S')\big)}\).
    Since no vertex in \(A\cup S'\) is adjacent to a vertex of \(G'\setminus \big(S'\cup N_{G'}(S')\big)\)
    and \(|A\cup S'| = |A| + |S'| \geq m_1\),
    \(A\cup S' \cup V(K')\) induces a copy of \(K_{m_1,\ldots,m_k}\) in \(\overline{G}\), as desired. 
    
    Therefore, we may assume that \(n-1 \geq |P| \geq n - m_2 - 2m_1\).
    Now, consider  \(G'=G\setminus V(P)\).
    Since  \(|G'|=N-|P|\geq (n-1)(k-2)+m_1\), 
    by the induction hypothesis, 
    either \(P_n \subseteq G'\) or \(K_{m_1,\ldots,m_{k-1}} \subseteq \overline{G'}\).
    In the former case, we have \(P_n\subseteq G\), as desired.
    Thus, we may assume that there is a copy \(K'\) of \(K_{m_1,\ldots,m_{k-1}}\) in \(\overline{G'}\). 
    Let \(K^*\) be the union of the components of \(G'\) that contain vertices of \(K'\).
    Since \(K^*\subseteq G' = G\setminus V(P)\), we have
    \begin{equation}\label{eq:gen.2}
        |K^*|+|P|\leq N.
    \end{equation}
    Moreover, if \(N-|K^*|-|N(K^*)|\geq m_k\), 
    then there is a copy of \(K_{m_1,\ldots,m_{k}}\) 
    in \(\overline{G}\) obtained from \(K'\) 
    by adding \(m_k\) vertices of \(V(G)\setminus \big(K^*\cup N(K^*)\big)\).
    Therefore, we may assume that 
    \begin{equation}\label{eq:gen.3}
        N-|K^*|-|N(K^*)|\leq m_k-1.
    \end{equation}

    Now, we use a simple induction to prove that for each \(i \geq 2\) we have \(m_i \geq 2(m_1 + \cdots + m_{i-1})\).
    Indeed, this holds for \(m_2\) since \(\varepsilon\cdot m_2 \geq 2m_1^2\geq 2m_1\).
    Now, since \(m_{i-1}\geq 2m_{i-2}^2 \geq 2m_{i-2} \geq 2\), 
    if \(m_{i-1} \geq 2(m_1 + \cdots + m_{i-2})\),
    we have \(\varepsilon \cdot m_i\geq 2m_{i-1}^2 \geq 2m_{i-1} + m_{i-1} \geq 2m_{i-1} + 2(m_1 + \cdots + m_{i-2}) = 2(m_1 + \cdots + m_{i-1})\),
    as desired.
    Therefore, we have \(m_i^2 \geq 2 m_i(m_1 + \cdots + m_{i-1})\).
    Summing over \(i \geq 2\), we have
    \begin{align}
        \varepsilon(m_1 + \cdots + m_k) 
                &= \sum_{i=1}^k \varepsilon \cdot m_i
        \geq \sum_{i=0}^{k-1} 2m_i^2 \nonumber\\
        & = \sum_{i=0}^{k-1} m_i^2 + \sum_{i=0}^{k-1}2m_i(m_1 + \cdots + m_{i-1})
        = (m_1 + \cdots + m_{k-1})^2 \label{eq:eps+sum}.
    \end{align}

    Summing~\eqref{eq:gen.2} and~\eqref{eq:gen.3},
    we obtain
    \begin{align*}
        |N(K^*)| 
                \geq |P| - m_k + 1
            &   \geq n - m_2 - 2m_1 - m_k + 1 \\
            &   \geq m_k + (m_1+\cdots+m_{k-1})(m_1+\cdots+m_{k-1}-1),
    \end{align*}
    where the last inequality follows by~\eqref{eq:eps+sum} because \(n\geq (2+\varepsilon)(m_1+\cdots + m_k)\).
    
    Finally, by Lemma~\ref{lemma:abab}
    given two distinct vertices \(u\) and \(v\) in \(K^*\),
    \(N_P(u)^-\) has at most one vertex in \(N(v)\).
    Consequently, \(N_P(u)^-\) has at most \(m_1+\cdots+m_{k-1}-1\) vertices of \(\cup_{v\in V(K^*)\setminus\{u\}} N(v)\).
    Thus, \(N_P(K^*)^- = \cup_{u\in V(K^*)}N_P(u)^-\) has at most \((m_1+\cdots+m_{k-1})(m_1+\cdots+m_{k-1}-1)\) neighbors of \(K^*\).
    Therefore,  there is a set \(N^* \subseteq N_P(K^*)^-\setminus N_P(K^*)\)
    with \(m_k\) vertices,
    and hence \(N^* \cup V(K^*)\) induces a copy of \(K_{m_1,\ldots,m_{k}}\) in \(\overline{G}\), as desired.
\end{proof}

\section{Concluding remarks}

In this paper we present a family of graphs \(H\) for which the family of
\(H\)-good paths is almost as large as possible.
We observe that the unbalance condition \(\varepsilon\cdot m_i \geq 2m_{i-1}^2\)
could be replaced by the slightly weaker condition \(\varepsilon (m_1 + \cdots + m_i) \geq (m_1 + \cdots + m_{i-1})(m_1 + \cdots + m_{i-1} - 1)\),
but this would require a longer checking on the induction hypothesis conditions,
while keeping a quadratic inequality.
Nevertheless, we believe that the results presented in Section~\ref{sec:auxiliary} could be deepened in order to improve the unbalance condition, perhaps to a subquadratic inequality.
For example, it's not hard to see the relation between Lemma~\ref{lemma:blue-clique} and Lemma~\ref{lemma:abab}. 
This connection suggests the existence of a more general result that considers a larger number of components of \(G\setminus V(P)\).

\bibliographystyle{amsplain}
\providecommand{\bysame}{\leavevmode\hbox to3em{\hrulefill}\thinspace}
\providecommand{\MR}{\relax\ifhmode\unskip\space\fi MR }
\providecommand{\MRhref}[2]{%
  \href{http://www.ams.org/mathscinet-getitem?mr=#1}{#2}
}
\providecommand{\href}[2]{#2}

\end{document}